\documentclass[12pt]{amsart}
\usepackage{amsmath,amssymb,amsthm,color, euscript, enumerate,comment}
\usepackage{dsfont}
\usepackage{longtable}
\usepackage[all]{xy}
\usepackage{mathtools} %

\newcommand{\maru}[1]{{\ooalign{\hfil#1\/\hfil\crcr
\raise.167ex\hbox{\mathhexbox20D}}}}

\newcommand{\ruby}[2]{%
 \leavevmode
 \setbox0=\hbox{#1}%
 \setbox1=\hbox{\tiny #2}%
 \ifdim\wd0>\wd1 \dimen0=\wd0 \end{lemma}se \dimen0=\wd1 \fi
 \hbox{%
   \kanjiskip=0pt plus 2fil
   \xkanjiskip=0pt plus 2fil
   \vbox{%
     \hbox to \dimen0{%
       \tiny \hfil#2\hfil}%
     \nointerlineskip
     \hbox to \dimen0{\mathstrut\hfil#1\hfil}}}}

\theoremstyle{plain}

\theoremstyle{definition}

\theoremstyle{remark}

\numberwithin{equation}{section}
\title[Generalised Checkerboard Lattices]{Generalised Checkerboard Lattices}
\author[A. Matsuo]{Atsushi Matsuo}%
\address[A. Matsuo]{Graduate School of Mathematical Sciences, The University of Tokyo, Tokyo   153-8914, Japan }
\email{matsuo@ms.u-tokyo.ac.jp}
\author[H. Shimakura]{Hiroki Shimakura}
\address[H. Shimakura]{Department of Applied Mathematics, Faculty of Sciences, Fukuoka University, Fukuoka 814-0180, Japan }%
\email {shimakura@fukuoka-u.ac.jp}%
\date{}
\thanks{H.\ Shimakura was partially supported by JSPS KAKENHI Grant Numbers JP19KK0065, JP20K03505 and JP24K06658.}
\makeatletter
\thm@style{plain}
\newtoks\thm@bodyfont  \thm@bodyfont{\slshape}
\newtoks\thm@headfont  \thm@headfont{\bfseries}
\newtoks\thm@notefont  \thm@notefont{}
\newtoks\thm@headpunct \thm@headpunct{.}
\newskip\thm@preskip \newskip\thm@postskip
\def\thm@space@setup{%
  \thm@preskip=\topsep \thm@postskip=\thm@preskip
}
\makeatother
\theoremstyle{plain}

\theoremstyle{remark}

\newtheorem{claim}{}
\def\claim[#1]{\global\def\thisclaim{#1}\csname \thisclaim\endcsname}
\def\endclaim{\csname end\thisclaim\endcsname}
\setbox0\hbox{$(1)$\hskip\labelsep}
\edef\theitemleftmargini{\the\wd0}
\let\olditemize\itemize\let\oldenditemize\enditemize
\def\itemize{\leftmargini\theitemleftmargini\olditemize\itemsep2pt\parskip0pt}
\def\enditemize{\oldenditemize}
\edef\theenumleftmargini{0pt}
\let\oldenumerate\enumerate\let\oldendenumerate\endenumerate
\makeatletter
\let\org@item\@item
\def\enumerate{\leftmargini\theenumleftmargini\oldenumerate\itemsep2pt\parskip0pt\leftmargin0pt\def\@item[##1]{\org@item[]\hskip0pt##1\hskip\labelsep}}
\makeatother
\def\endenumerate{\oldendenumerate}
\makeatletter%
\def\@listi{\leftmargin\leftmargini%
              \topsep \z@%
              \parsep \z@%
              \itemsep \parsep}%
\def\@listI{\leftmargin\leftmargini%
  \topsep 0.15\baselineskip%
\relax}%
\def\@listii{\leftmargin\leftmarginii%
  \topsep 0.1\baselineskip%
\relax}%
\makeatother%
\makeatletter%
\def\eqalign{\@ifnextchar<{\eqalign@a}{\eqalign@a<0pt>}}
\def\eqalign@a<#1>#2{%
\vcenter{%
\vskip0.2ex%
\addtolength{\baselineskip}{#1}%
\mathsurround=0pt\ialign{%
$\displaystyle##$\hfil&&$$$\displaystyle{}##$\hfil\crcr\mathstrut
\crcr\noalign{\kern-\baselineskip}#2\crcr\mathstrut\crcr
\noalign{\kern-\baselineskip}%
}%
\vskip0.2ex%
}%
}%
\makeatother%
\def\now{{%
   \def\Time{3}%
   \def\Hour{4}%
   \def\Minute{5}%
   \count\Time=\time\relax
   \ifnum\count\Time=0
      \count\Time=1440 
   \fi
   \count\Hour=\count\Time\relax
   \divide\count\Hour by 60\relax
   \count\Minute=\count\Hour\relax
   \multiply\count\Minute by -60\relax
   \advance\count\Minute by \count\Time\relax 
   \ifnum\count\Hour=0
      \count\Hour=12
   \else
   \fi
   \the\count\Hour\relax:%
   \ifnum\count\Minute<10
      0%
   \fi
   \the\count\Minute\relax
}}%
\def\[{\begin{equation}}
\def\]{\end{equation}}
\mathtoolsset{showonlyrefs=true}
\begin{document}
%


\begin{abstract}\baselineskip15pt
\noindent
A series of integral lattices parametrised by integers $k,m,n$ are introduced and investigated, where $n$ is the rank of the lattice, including the root lattices described in a uniform way and unimodular lattices such as the Niemeier lattices of type $A_{24}$ and $D_{24}$. 
The lattices are characterised by means of a sublattice isomorphic to the root lattice of type $A_{n-1}$. 
A sufficient condition for existence of an orthogonal $k$-frame of such a lattice is given in terms of symmetric $2$-designs.
\end{abstract}
%
\maketitle

\section{Introduction}
\label{section:Introduction}
Let $k,m,n$ be integers with $m\ne 0$ and $n\geq 1$.
Consider the symmetric bilinear form $(\ |\ )_{k,m,n}$ on $\mathbb{Q}^n$ defined by the matrix
\[
\label{eq:240315-03}
B_{k,m,n}=I_n+\frac{k-m}{m^2}J_n, 
\]
where $I_n$ is the identity matrix of order $n$ and $J_n$ the all 1 matrix of the same order. 
Let $\mathrm{lat}(x)$ denote the sum of the entries of $x\in\mathbb{Q}^n$, which we call the {\it latitude}\/ of $x$. 
Then, for $x,y\in\mathbb{Q}^n$, 
\[
\label{eq:240305-01}
(x|y)_{k,m,n}=x\cdot y+\frac{\mathrm{lat}(x)\,\mathrm{lat}(y)}{m^2}(k-m)
,\]
where $x\cdot y$ denotes the Euclidean inner product of $x,y\in\mathbb{Q}^n$. 
We will denote the squared norm $(x|x)_{k,m,n}$ by $|x|_{k,m,n}^2$.
We sometimes omit writing the subscripts for short. 

Consider the subgroup $L_{k,m,n}$ of the additive group $\mathbb{Z}^n\subset \mathbb{Q}^n$ consisting of the elements of latitude divisible by $m$:
\[
L_{k,m,n}=\{x\in \mathbb{Z}^n\,|\,\mathrm{lat}(x)\in m\mathbb{Z}\}
.\]
It is clear by \eqref{eq:240305-01} that the free abelian group $L_{k,m,n}$ equipped with $(\ |\ )_{k,m,n}$ is an integral lattice.

For $h\in m\mathbb{Z}$, let $L_{k,m,n}[h]$ denote the subset of $L_{k,m,n}$ consisting of the elements of latitude $h$, so that
\[
L_{k,m,n}=\bigsqcup_{m|h}L_{k,m,n}[h]
.\]
The subset $M=L_{k,m,n}[0]$ of latitude $0$ is a sublattice of $L_{k,m,n}$, 
which we will call the {\it base lattice}.

Let $e_1,\ldots,e_n$ be the standard basis of $\mathbb{Q}^n$. 
Then the base lattice $M$ is isomorphic to the root lattice $A_{n-1}$ generated by the simple roots $e_2-e_1,\ldots,e_{n}-e_{n-1}$.
The action of $M$ on $L_{k,m,n}[h]$ by addition is clearly transitive, thus $L_{k,m,n}[h]=M+x$ for any element $x$ of latitude $h$.
In particular, any element of latitude $m$ together with $e_2-e_1,\ldots,e_{n}-e_{n-1}$ form a basis of the lattice $L_{k,m,n}$. 

For instance, we may choose $me_1$ as the element of latitude $m$.
Then the Gram matrix with respect to the basis $me_1,e_2-e_1,\ldots,e_{n}-e_{n-1}$ reads
\[
G_{k,m,n}
=\begin{bmatrix}
m^2-m+k\ 
&-m&&&\cr
-m\ &\phantom{-}2&-1&&\cr
&-1&\phantom{-}2&-1&\cr
&&\ddots&\ddots&\ddots\cr
&&&-1&\phantom{-}2&-1\ \cr
&&&&-1&\phantom{-}2\ \cr
\end{bmatrix}.
\]
Note that $|me_1|^2=m^2-m+k$ is even if and only if $k$ is even.

Let us now assume $1\leq m\leq n$.
Then we may choose $e_1+\cdots+e_m$ as the element of latitude $m$, which satisfies
\[
|e_1+\cdots+e_m|^2=k,\ \ (e_{i+1}-e_i|e_1+\cdots+e_m)=-\delta_{i,m}\ \ (1\leq i\leq n-1)
.\]
Therefore, if $1\leq m\leq n-1$, then the elements form the nodes of the (generalized) Coxeter-Dynkin diagram
\[
\xymatrix@M0pt@R0pt@C0pt{
&&&&&&\hskip-2em \scriptstyle e_1+\cdots+e_m\vcenter to 2ex{}\hskip-2em \cr
&&&&&&
\circ\ar@{-}[dd]
\cr
\vcenter to 2em{}\cr
\circ\ar@{-}[rr]&\hskip3em&\circ\ar@{-}[rr]&\hskip3em&\ \cdots\ \ar@{-}[rr]&\hskip3em&\circ\ar@{-}[rr]&\hskip3em&\ \cdots\ \ar@{-}[rr]&\hskip3em&\circ\ar@{-}[rr]&\hskip3em&\circ\cr
\hskip-2em {\scriptstyle e_2-e_1}\hskip-2em &&\hskip-2em {\scriptstyle e_3-e_2}\hskip-2em &&&&\hskip-2em {\scriptstyle e_{m+1}-e_{m}}\hskip-2em &&&&\hskip-2em {\scriptstyle  e_{n-1}-e_{n-2}}\hskip-2em &&\hskip-1.5em {\scriptstyle  e_{n}-e_{n-1}}\hskip-2.5em 
\cr
\vcenter to 3ex{}\cr}
\]
while for $m=n$, the lattice $L_{k,n,n}$ is isomorphic to the orthogonal sum of the base lattice $M\cong A_{n-1}$ and a one-dimensional lattice generated by $e_1+\cdots+e_n$.

For $k=2$, the lattices $L_{2,1,n}$ are the $A_{n}$ lattices described in an unusual way, and the lattices $L_{2,2,n}$ for $n\geq 2$ are the $D_{n}$ lattices, where $D_2=A_1\oplus A_1$ and $D_3=A_3$. 
Moreover, the lattices $L_{2,3,n}$ agree with the $E_{n}$ lattices for $3\leq n\leq 8$, where $E_3=A_2\oplus A_1$, $E_4=A_4$, $E_5=D_5$, thus giving rise to a uniform construction of the root lattices:
\[
\label{eq:240312-01}
L_{2,1,n}=A_n\ (1\leq n),\ \ 
L_{2,2,n}=D_n\ (2\leq n),\ \ 
L_{2,3,n}=E_n\ (3\leq n\leq 8).
\]
The construction of $L_{2,3,n}=E_n$ for $n=6,7,8$ is seen to be the dual of the uniform construction of $E_6,E_7,E_8$ by Shioda \cite{Shioda} (cf.\ Note \ref{claim:240318-03}). 

The lattice $L_{k,m,n}$ is worth calling the {\it generalised checkerboard lattice}, for the lattices $L_{2,2,n}=D_n$ are often called the checkerboard lattice, among which the lattice $D_2=A_1\oplus A_1$ is the ordinary checkerboard.

By construction, the lattice $L_{k,m,n}$ contains a primitive sublattice of rank $n-1$ isomorphic to the root lattice $A_{n-1}$.
We  will show in Sect.\ 3 that this property characterises the lattices $L_{k,m,n}$ (Theorem 1), after describing some basic properties of $L_{k,m,n}$ in Sect.\ 2. 

In Sect.\ 4, we will examine the cases when the lattice $L_{k,m,n}$ becomes a root lattice (Subsect.\ 4.1) and those when it is unimodular (Subsect.\ 4.2).
For example, the Niemeier lattices (\cite{Niemeier}) of type $A_{24}$ and $D_{24}$ arise as $N(A_{24})=L_{4,5,24}$ and $N(D_{24})=L_{6,11,24}$, respectively. 
We will also describe a particular setting under which the lattice $L_{k,m,n}$ carries an orthogonal $k$-frame (Subsect.\ 4.3).
In particular, any symmetric $2$-$(n,m,\lambda)$ design gives rise to an orthogonal $(m-\lambda)$-frame of $L_{m-\lambda,m,n}$ if $1\leq m\leq n-1$ (Proposition \ref{claim:240326-01}). 

In Appendix A, we list up the roots when $L_{2,m,n}$ is a root lattice and calculate the orders of the Weyl groups as an application. 

The existence of a sublattice of type $A_{n-1}$ in our construction, although not a full sublattice, is in contrast with the existence of a sublattice of type $A_1^{n}=\sqrt{2}\mathbb{Z}^n$ in Construction A of Leech-Sloane \cite{LeechSloane}, a construction of integral lattices from binary codes.

We note that our construction can be generalised at least in two directions: one is to adjoin codes over the ring $\mathbb{Z}/r\mathbb{Z}$ for an integer $r\geq 2$, and the other is to replace $A_{n-1}$ or the ambient lattice $\mathbb{Z}^n$ with other lattices. 
A recent work of Shimada \cite{Shimada} seems to fit in a position of such generalisation, as Corollary 1.1 of Ref.\ \cite{Shimada} certainly shares the same feature with ours. 

After finishing this work, the authors came to know that the construction for $k=2$ has already been considered by Jensen et al.\ \cite{JensenEtAl} and further studied by Baur et al.\ \cite{BaurEtAl2023}. 

We refer the reader to \cite{ConwaySloane}, \cite{Ebeling}, \cite{MilnorHusemoller}, and \cite{Serre} for integral lattices, to \cite{Bourbaki} and \cite{Humphreys} for root systems, and to \cite{BethEtAl} and \cite{Lander} for block designs. 
%
\section{Basic properties of the lattices}
\label{section:BasicProperties}
%
%
Let $k,m,n$ be integers with $m\ne 0$ and $n\geq 1$ and consider the lattice $L_{k,m,n}$.
%
%
\subsection{Determinants and signatures}
\label{subsection:Determinants}
The determinant of the matrix $B_{k,m,n}$ of the bilinear form $(\ |\ )_{k,m,n}$ on $\mathbb{Q}^n$ is easily calculated, and the result reads
\[
\det B_{k,m,n}=\det \Bigl(I_n+\frac{k-m}{m^2}J_n\Bigr)
=1+\frac{(k-m)n}{m^2}
.\]
Since $|\mathbb{Z}^n/L_{k,m,n}|=m$, the determinant (or the discriminant) of the lattice $L_{k,m,n}$ is given by $m^2\det B_{k,m,n}$, thus
\[
\det L_{k,m,n}
=m^2-mn+kn
.\]
As the base lattice $M=L_{k,m,n}[0]\cong A_{n-1}$ is positive-definite of rank $n-1$, the signature of $L_{k,m,n}$ is determined by the sign of $\det L_{k,m,n}$.
\begin{claim}[Proposition]
\sl
Let $k,m,n$ be integers with $m\ne 0$ and $n\geq 1$. 
\begin{itemize}
\item[\rm(1)]
If $m^2-mn+kn>0$, then the signature of $L_{k,m,n}$ is $(n,0,0)$.
\item[\rm(2)]
If $m^2-mn+kn=0$, then the signature of $L_{k,m,n}$ is $(n-1,0,1)$.
\item[\rm(3)]
If $m^2-mn+kn<0$, then the signature of $L_{k,m,n}$ is $(n-1,1,0)$.
\end{itemize}
\end{claim}
In other words, $L_{k,m,n}$ is positive-definite, positive-semidefinite with one-dimensional radical, and hyperbolic if and only if $m^2-mn+kn$ is positive, zero, and negative, respectively.

The parameters $k,m,n$ for which the number $m^2-mn+kn$ coincides with a given integer $d$ are characterised as follows.

\begin{claim}[Lemma]
\label{claim:240310-01}
\sl
Let $d,k,m,n$ be integers.
Then $m^2-mn+kn=d$ if and only if there exist integers $p,q$ satisfying
\[
\label{eq:240310-01}
k^2-d=pq,\ \ 
m=k+p,\ \ 
n=2k+p+q.
\]
\end{claim}
\begin{proof}[Proof.]
It is easy to check that \eqref{eq:240310-01} implies $m^2-mn+kn=d$.
Conversely, let $d,k,m,n$ be integers and assume $m^2-mn+kn=d$. 
Since $m$ is an integer, the discriminant $n^2-4kn+4d$ is a square of an integer $s\geq 0$.
Moreover, since $n$ is an integer, $4k^2-4d+s^2$ is a square of an integer $t\geq 0$.
For such $s$ and $t$, we have
\[
\label{eq:240310-02}
4k^2-4d=t^2-s^2
,\ \ 
2m=2k\pm t\pm s 
,\ \ 
n=2k\pm t
.\]
Let $p$ and $q$ be the halves of $t+s$ and $t-s$, respectively.
They are integers by \eqref{eq:240310-02}, and any change of signs of $s$ and $t$ is achieved by replacing $(p,q)$ with one of $\pm(p,q),\pm(q,p)$, all of which are solutions $(x,y)$ of the same equation $k^2-d=xy$. 
We may therefore assume 
that the signs in \eqref{eq:240310-02} are all positive, and the conditions \eqref{eq:240310-02} reduce to \eqref{eq:240310-01} as desired.
\end{proof}
Applying Lemma \ref{claim:240310-01} to the parameters of the lattice $L_{k,m,n}$, we immediately obtain that $\det L_{k,m,n}=d$ if and only if there exist integers $p,q$ satisfying the conditions \eqref{eq:240310-01}.
\begin{claim}[Remark]
\label{claim:240315-01}
If $\det L_{k,m,n}=d$, then $n^2-4kn+4d\geq 0$, hence $k\leq n/4+d/n$ if $n\ne 0$.
In particular, if $(n,d)=(8,1)$, $(7,2)$, or $(6,3)$, for instance, then $k\leq 2$. 
\end{claim}
%
\subsection{The opposite lattices}
\label{subsection:Opposite}
\label{240318-09}
It is clear from the diagram in Sect.\ 1 that the lattices $L_{k,m,n}$ and $L_{k,n-m,n}$ are isomorphic to each other if $1\leq m\leq n-1$. 

To be more precise, let $k,m,n$ be arbitrary integers with $m\ne 0$ and $n\geq 1$. 
Consider the linear map $\theta_{m,n}\colon \mathbb{Q}^n\to \mathbb{Q}^n$ defined by setting 
\[
\label{eq:240318-05}
\theta_{m,n}(x)=x-\frac{\mathrm{lat}(x)}{m}(e_1+\cdots+e_n)
.\]
In particular, $\theta_{m,n}(e_i-e_j)=e_i-e_j$ for all $i\ne j$ and 
\[
\label{eq:240318-07}
\theta_{m,n}(e_1+\cdots+e_n)=-\frac{n-m}{m}(e_1+\cdots+e_n)
.\]
Since $e_1+\cdots+e_n,e_2-e_1,\ldots,e_n-e_{n-1}$ form a basis of $\mathbb{Q}^n$, the map $\theta_{m,n}$ is a linear isomorphism with the inverse given by $\theta_{n-m,n}$ unless $m=n$.

Now the latitudes of elements of $\mathbb{Q}^n$ behave under $\theta_{m,n}$ as
\[
\mathrm{lat}(\theta_{m,n}(x))=-\frac{\mathrm{lat}(x)}{m}(n-m)
.\]
Thus $\mathrm{lat}(x)\in m\mathbb{Z}$ implies $\mathrm{lat}(\theta_{m,n}(x))\in (n-m)\mathbb{Z}$.
Therefore, if $m\ne n$, then the map $\theta_{m,n}$ restricts to a homomorphism of abelian groups between $L_{k,m,n}$ and $L_{k,n-m,n}$, and it follows from \eqref{eq:240305-01} and \eqref{eq:240318-05} that
$(\theta_{m,n}(x)|\theta_{m,n}(y))_{k,n-m,n}
=(x|y)_{k,m,n}
$
holds for all $x,y\in\mathbb{Q}^n$.

We have thus obtained the following.
\begin{claim}[Proposition]
\sl
Let $k,m,n$ be integers with $m\ne 0,n$ and $n\geq 1$. 
Then $\theta_{m,n}$ induces an isomorphism of lattices between $L_{k,m,n}$ and $L_{k,n-m,n}$. 
\end{claim}
We will say that the lattice $L_{k,n-m,n}$ is {\it opposite}\/ to $L_{k,m,n}$. 
The proposition says that a pair of opposite lattices are isomorphic to each other.
For example, if $1\leq m\leq n-1$, then  
\[
\theta_{m,n}(e_{1}+\cdots+e_{m})=-(e_{m+1}+\cdots+e_{n})
.\]

\begin{claim}[Remark]
\label{claim:240324-02}
\begin{enumerate}
\item[(1)]
If $\det B\ne 0$, then $e_1+\cdots+e_n$ spans the subspace $({M}^{\perp})_{\mathbb{Q}}$ orthogonal to the $\mathbb{Q}$-span $M_{\mathbb{Q}}$ of the base lattice $M=L[0]$, and the properties $\theta_{m,n}=1$ on $M_{\mathbb{Q}}$ and $\theta_{m,n}=-(n-m)/m$ on $({M}^{\perp})_{\mathbb{Q}}$ characterise the map $\theta_{m,n}$.
\item[(2)]
Since the bilinear form is stable under the change of parameters given by $(k,m)\leftrightarrow(k-2m,-m)$, that is, $(\ |\ )_{k,m,n}=(\ |\ )_{k-2m,-m,n}$, the lattices $L_{k,m,n}$ and $L_{k-2m,-m,n}$ are identical.
Therefore, the parameters of the lattice $L_{k,m,n}$ are put in a form satisfying $1\leq m\leq n$ by repeatedly applying $L_{k,m,n}=L_{k-2m,-m,n}$ and $L_{k,m,n}\cong L_{k,n-m,n}$. 
If $1\leq m\leq n-1$, we may further assume $1\leq m\leq n-m$.
\end{enumerate}
\end{claim}
%
\subsection{The dual lattices}
\label{subsection:240318-01}
Recall that the {\it dual lattice}\/ of an integral lattice $L\subset \mathbb{Q}^n$ is the free abelian group
$L^*=\{x\in {\mathbb{Q}^n}\,|\,(x|L)\subset\mathbb{Z}\}$
equipped with the restriction of the $\mathbb{Q}$-linear extension of the bilinear form of $L$, which need not be an integral lattice, but a {\it rational}\/ lattice. 

Recall the matrix $B=B_{k,m,n}$ of the bilinear form $(\ |\ )_{k,m,n}$ on $\mathbb{Q}^n$ defined by \eqref{eq:240315-03}. 
Assume $\det B\ne 0$, that is, the bilinear form is nondegenerate.
Identify $\mathbb{Q}^n$ with its dual space via the bilinear form.

Let $e^1,\ldots,e^n$ be the dual basis of the standard basis $e_1,\ldots,e_n$ of $\mathbb{Q}^n$. 
The two bases are related as
\[
e_i=\sum_{j=1}^{n}(e_i|e_j)e^j
\hbox{\ \ and\ \ }
e^i=\sum_{j=1}^{n}(e^i|e^j)e_j
.\]
Here the coefficients of the second equalities are the entries of the inverse matrix $C=B^{-1}$, given by
\[
\label{eq:240318-02}
(e^i|e^j)
=\delta_{i,j}-\frac{k-m}{m^2-mn+kn}
=\delta_{i,j}+\frac{m-k}{\det L_{k,m,n}}
.\]
Consider the following vector in $\mathbb{Q}^n$:
\[
\label{eq:240318-08}
e^0=\frac{1}{m}(e^1+\cdots+e^n)
.\]
Then, for $x=(x_1,\ldots,x_n)\in\mathbb{Z}^n\subset\mathbb{Q}^n$, 
\[
\label{claim:240317-01}
(x|e^0)=\frac{\mathrm{lat}(x)}{m},\ \ 
(x|e^i)=x_i\ \ (1\leq i\leq n)
.\]
Therefore, it is immediate that the dual lattice $L_{k,m,n}^*$ includes the $\mathbb{Z}$-span of $e^0,e^1,\ldots,e^n$, and it is easy to show the opposite inclusion as well.
We thus obtain the following.  
\begin{claim}[Proposition]
\sl
Let $k,m,n$ be integers with $m\ne 0$ and $n\geq 1$.
If $\det B\ne 0$, then $e^0,e^1,\ldots,e^n$ generate the dual lattice $L_{k,m,n}^*$.
\end{claim}
\begin{claim}[Remark]
\label{claim:240402-01}
\begin{enumerate}
\item[(1)]
If $1\leq m\leq n-1$, then $(e^0|e_{i+1}-e_{i})=0$
($1\leq i\leq n-1$) and $(e^0|e_1+\cdots+e_m)=1$, whence for $k=2$ the vector $e^0$ is the fundamental weight corresponding to the node $e_1+\cdots+e_m$ of the diagram in Sect.\ 1.
\item[(2)]
If $1\leq m\leq n$, then the element $z_j=(m-k)(e_1+\cdots+e_n)+(\det L)e_j$ of $L=L_{k,m,n}$ satisfies $(e_i|z_j)=(\det L)\delta_{i,j}$ ($1\leq i,j\leq n$).
Hence $w_n=(\det L)^{-1}z_n$ is the fundamental weight corresponding to the node $\alpha_{n-1}=e_n-e_{n-1}$ if $k=2$, $1\leq m\leq n-1$, and $\det L\ne 0$.
\end{enumerate}
\end{claim}
%
\section{Characterisation of the lattices}
\label{section:Characterisation}
Recall that a sublattice $M$ of an integral lattice $L$ is said to be {\it primitive}\/ if $L/M$ is torsion free, or equivalently, if $L\cap M_{\mathbb{Q}}=M$, where $M_{\mathbb{Q}}$ is the $\mathbb{Q}$-span of the sublattice $M$.
We will say that a lattice is {\it nondegenerate}\/ if the bilinear form is nondegenerate.
%
\subsection{Statement of the result}
\label{subsection:240307-02}
Consider the base lattice $M=L[0]$ of the lattice $L=L_{k,m,n}$.
Then, by construction, it satisfies the following properties: 
\begin{itemize}
\item[(i)]
$M$ is a primitive sublattice of rank $n-1$.
\item[(ii)]
$M$ is isomorphic to the root lattice $A_{n-1}$.
\end{itemize}
Existence of such a sublattice $M$ characterises the lattices $L_{k,m,n}$ as follows.
\begin{claim}[Theorem]
\label{claim:240316-06}
\sl
Let $L$ be an integral lattice of rank $n$.
Then the following conditions are equivalent to each other:
\rm
\begin{itemize}
\item[(1)]
\sl
$L$ is isomorphic to $L_{k,m,n}$ for some integers $k,m$ with $1\leq m\leq n$. 
\rm
\item[(2)]
\sl
$L$ contains a sublattice $M$ satisfying the conditions {\rm (i)} and {\rm (ii)}.
\end{itemize}
\end{claim}
The proof will be given in Subsect.\ \ref{subsection:240307-01}. 
\begin{claim}[Remark]
\label{claim:240317-02}
Let $k,m,n$ be integers with $m\ne 0$ and $n\geq 1$.
Then, by Theorem 1, there exist some integers $k',m'$ with $1\leq m'\leq n$ satisfying $L_{k,m,n}\cong L_{k',m',n}$ (cf.\ Remark \ref{claim:240324-02} (2)). 
\end{claim}
%
%
\subsection{Preliminaries on lattices}
\label{subsection:Preliminaries}
Let $L$ be an integral lattice of rank $n$ and $M$ a sublattice of $L$.
Let ${{M}^{\perp}}$ denote the annihilator of $M$ in $L$ defined as
\[
\label{eq:240401-01}
{{M}^{\perp}}=\{x\in L\,|\,(x|M)=0\}.
\]
Then $M$ is nondegenerate if and only if $L_{\mathbb{Q}}$ decomposes into the orthogonal sum of $M_{\mathbb{Q}}$ and $({M}^{\perp})_{\mathbb{Q}}$. 

Assume that $M$ is nondegenerate and consider the projections $i_M \colon L \to M_{\mathbb{Q}}$ and $i_{{M}^{\perp}}\colon L\to ({M}^{\perp})_{\mathbb{Q}}$ according to the orthogonal decomposition.
Denote their values at $x$ by $x_M$ and $x_{{M}^{\perp}}$. 
These maps induce maps
\[
j_M\colon L/(M+{{M}^{\perp}})\to M_{\mathbb{Q}}/M
\hbox{\ \ and\ \ }
j_{{M}^{\perp}}\colon L/(M+{{M}^{\perp}})\to ({M}^{\perp})_{\mathbb{Q}}/{{M}^{\perp}}
,\] 
respectively, which are group homomorphisms. 
Let $M^*$ be the dual lattice of $M$.
\begin{claim}[Lemma]
\label{claim:240306-02}
\sl
\begin{itemize}
\item[\rm(1)]
\sl
The map $i_M$ is valued in $M^*$.
\item[\rm(2)]
The map $j_M$ is injective.
\item[\rm(3)]
\sl
If\/ $M$ is primitive, then the map $j_{{M}^{\perp}}$ is injective. 
\end{itemize}
\end{claim}
\begin{proof}[Proof.]
For $x\in L$, we have $x_M\in M^*$ by $(x_M|M)=(x-x_{{M}^{\perp}}|M)=(x|M)\subset \mathbb{Z}
$.
The map $j_M$ is injective since if $x_M\in M$, then $x_{{M}^{\perp}}=x-x_M\in L\cap ({M}^{\perp})_{\mathbb{Q}}={{M}^{\perp}}$, thus $x=x_M+x_{{M}^{\perp}}\in M+ {{M}^{\perp}}$. 
If $M_{\mathbb{Q}}\cap L=M$, then $j_{{M}^{\perp}}$ is also injective by the same reason.
\end{proof}
\begin{claim}[Lemma]
\label{claim:240306-03}
\sl
If $M$ is primitive and $\mathrm{rank}\, {{M}^{\perp}}=1$, then $L/(M+{{M}^{\perp}})$ is cyclic.
\end{claim}
\begin{proof}[Proof.]
By Lemma \ref{claim:240306-02}\ (3), the group homomorphism $j_{{M}^{\perp}}$ is injective.
So the group $L/(M+{{M}^{\perp}})$ is cyclic as its image under the injective homomorphism $j_{{M}^{\perp}}$ is a finite subgroup of $({M}^{\perp})_{\mathbb{Q}}/{{M}^{\perp}}\cong \mathbb{Q}/\mathbb{Z}$. 
\end{proof}
\begin{claim}[Remark]
\label{claim:240316-07}
By Lemma \ref{claim:240306-02}, the group $L/(M+{{M}^{\perp}})$ is identified with a subgroup of $M^*/M$.
Hence $|L/(M+{{M}^{\perp}})|$ divides $|M^*/M|$, and if $M^*/M$ is cyclic, then $L/(M+{{M}^{\perp}})$ is also cyclic.
\end{claim}
%
%
\subsection{Proof of Theorem 1}
\label{subsection:240307-01}
It remains to show that (2) implies (1).
Let $L$ be an integral lattice of rank $n$ and $M$ a sublattice satisfying (i) and (ii), that is, $M$ is a primitive sublattice of rank $n-1$ and isomorphic to the root lattice $A_{n-1}$.
We are to show $L\cong L_{k,m,n}$ for some integers $k$ and $m$. 
\begin{proof}[Proof.]
Since $M$ is nondegenerate by (ii),  $L_{\mathbb{Q}}=M_{\mathbb{Q}}\oplus ({M}^{\perp})_{\mathbb{Q}}$.
In particular, $\mathrm{rank}\, {{M}^{\perp}}=1$.
If $L=M+{{M}^{\perp}}$, then $L\cong L_{k,m,n}$ with $k=|\nu |^2$ and $m=n$ for a generator $\nu $ of ${{M}^{\perp}}$; we may and do assume $L\ne M+{{M}^{\perp}}$. 

By Lemma \ref{claim:240306-03}, the group $L/(M+{{M}^{\perp}})$ is cyclic, say of order $d\geq 2$. 
Let $\lambda$ be a representative in $L$ of a generator of $L/(M+{{M}^{\perp}})$ and $\nu $ a generator of ${{M}^{\perp}}$.
By Lemma \ref{claim:240306-02} (3), $\lambda_{{M}^{\perp}}=(c/d)\nu $ for some $c\in\mathbb{Z}$ with $\mathrm{gcd}(c,d)=1$.
Choose $a,b\in\mathbb{Z}$ such that $ac+bd=1$. 
Then $L=a\lambda+M+M^\perp$ by $\mathrm{gcd}(a,d)=1$.
Since $a\lambda_M\in M^*$ by Lemma \ref{claim:240306-02} (1), there exists $\mu \in M$ such that $a\lambda_M+\mu$ becomes one of the fundamental weights with respect to any prescribed choice of the simple roots $\alpha_1,\ldots,\alpha_{n-1}$ of $M\cong A_{n-1}$ ordered in accordance with $e_1-e_2,\ldots,e_{n-1}-e_n$ as usual.

Let $k$ be the squared norm of $\beta=a\lambda+\mu+b\nu\in L$ and $m$ the position of the simple root corresponding to the fundamental weight $\beta_M=a\lambda_M+\mu$.
Then $L\cong L_{k,m,n}$ since $\alpha_1,\ldots,\alpha_{n-1}$ and $\beta$ form a basis of $L$ and the two lattices have the same Gram matrix. 
To see that $M$ and $\beta$ generate $L$, note $d\beta_{{M}^{\perp}}=ad\lambda_{{M}^{\perp}}+bd\nu =ac\nu +bd\nu =\nu $.
Since $d\beta_{M}=ad\lambda_{M}+d\mu\in M$, we have $d\beta_{{M}^{\perp}}=d\beta-d\beta_{M}\in \mathbb{Z}\beta+M$. 
So ${{M}^{\perp}}=\mathbb{Z}\nu =\mathbb{Z}d\beta_{{M}^{\perp}}\subset \mathbb{Z}\beta+M$, thus $L=\mathbb{Z}\beta+M+{{M}^{\perp}}=\mathbb{Z}\beta+M$.
\end{proof}
\begin{claim}[Remark]
\label{claim:240306-04}
By the same proof as above, the following statement holds true. 
Let $L$ be an integral lattice of rank $n$, $M$ a primitive sublattice of rank $n-1$ and $\alpha_1,\ldots,\alpha_{n-1}$ a basis of $M$.
If $M$ is nondegenerate, then there exists an element $\beta\in L$ such that $L=\mathbb{Z}\beta+M$, $\beta_M\in M^*$, and $(\beta_M|\alpha_i)=\delta_{i,m}$ ($1\leq i\leq n-1$) for some $m$ with $1\leq m\leq n$.
We note that if $M$ is a root lattice, $\alpha_1,\ldots,\alpha_{n-1}$ the simple roots,  and $1\leq m\leq n-1$, then $\beta_M$ is a minuscule weight. (For $A_{n-1}$, every fundamental weight is minuscule.)
\end{claim}
%
\section{Examples and applications}
\label{section:ExamplesAndApplications}
Let $k,m,n$ be integers with $m\ne 0$ and $n\geq 1$.
In classifying the lattices $L_{k,m,n}$, we may assume without loss of generality that the parameters satisfy $1\leq m\leq n$ by Remark \ref{claim:240324-02} (2). 
If $1\leq m\leq n-1$, we may further assume $m\leq n-m$.
%
\subsection{Root lattices}
\label{subsection:RootLattices}
By a {\it root lattice}, we mean a positive-definite lattice generated by its roots, the elements of squared norm $2$. 
\begin{claim}[Lemma]
\label{claim:240315-02}
\sl
Let $k,m,n$ be integers with $1\leq m\leq n$.
Then $L_{k,m,n}$ is a root lattice only if $k=2$.
\end{claim}
\begin{proof}
Assume that $L=L_{k,m,n}$ is a root lattice and let $M=L[0]$ be the base lattice.
Since $1\leq m\leq n$, we must have $k=|e_1+\cdots+e_m|^2\geq 2$.
Assume $k\geq 4$. 
If $x\in L[m]$, then $|x|^2\geq k\geq 4$.
Therefore, $L[m]$ has no root. 
Since $L$ is generated by roots, there exists $h>m$ such that $L[h]$ has a root $\alpha$.
Then the full sublattice $M+\mathbb{Z}\alpha$ is a root lattice containing $A_{n-1}$.
Such a sublattice exists only when $L$ is of type $E_n$ with $n=6,7,8$ (cf.\ \cite{Martinet}), for which only $k=2$ is allowed by Remark \ref{claim:240315-01}. 
\end{proof}
The parameters for which $L_{k,m,n}$ becomes a root lattice are classified by the following proposition under the assumption $1\leq m\leq n-1$, while $L_{2,n,n}=A_{n-1}\oplus A_1$ for $m=n$. 
\begin{claim}[Proposition]
\sl
Let $k,m,n$ be integers with $1\leq m\leq n-m$.
Then $L_{k,m,n}$ is a root lattice if and only if $k=2$ and one of the following conditions is satisfied:
\begin{itemize}
\item[\rm(1)]
$m=1$ and\/ $2\leq n$.
\ \rm(2)\hskip\labelsep
$m=2$ and\/ $4\leq n$.
\ \rm(3)\hskip\labelsep
$m=3$ and\/ $6\leq n\leq 8$.
\end{itemize}
\end{claim}
\begin{proof}
By Lemma \ref{claim:240315-02}, the condition $k=2$ is necessary. 
It is easy to check `if part' by consulting the diagram in Sect.\ 1.
The proof of `only if' part is also easy by positivity of $\det L_{2,m,n}=m^2-mn+2n$.
\end{proof}
The three cases correspond to $A_n$, $D_n$, and $E_n$, respectively, as mentioned in Sect.\ 1.
See Appendix for the lists of roots.

\begin{claim}[Note]
\label{claim:240318-03}
\begin{enumerate}
\item[(1)]
For $L_{2,3,n}=E_n$ ($n=6,7,8$), choose the simple roots $e_{1}-e_{2},\ldots,e_{n-1}-e_{n},-e_1-e_2-e_3$.
Recall the dual basis $e^1,\ldots,e^n$ of $e_1,\ldots,e_n$ and the element $e^0=(1/3)(e^1+\cdots+e^n)$ in Subsect.\ \ref{subsection:240318-01}. 
Then
\[
(e^i|e^j)
=
\delta_{i,j}+\frac{1}{9-n}\ \ (1\leq i,j\leq n)
,\]
and the simple roots are expressed as
$
e^{1}-e^{2},\ldots,e^{n-1}-e^{n},\ e^0-e^1-e^2-e^3
$
in agreement with the description of $E_6,E_7,E_8$ by Shioda \cite{Shioda} (cf.\ \cite{Manin},\cite{SchuttShioda}) under identification of $e^0,e^1,\ldots,e^n$ with $v_0,u_1,\ldots,u_n$ in Ref.\ \cite{Shioda}.
\item[(2)]
Let $B_n,C_n,F_n$, and $G_n$ denote the lattices generated by the root systems of the corresponding types, respectively, normalised so that the long roots have squared norm $2$.
These lattices are also realised as $L_{k,m,n}$ up to rescaling: $B_n=L_{1,1,n}=\mathbb{Z}^n$, $\sqrt{2}C_n=L_{4,1,n}$, $\sqrt{2}F_4=L_{2,2,4}=D_4$, and $\sqrt{3}G_{2}=L_{2,1,2}=A_{2}$. 
\end{enumerate}
\end{claim}%
%
%
\subsection{Unimodular lattices}
\label{subsection:UnimodularLattices}
Recall that a lattice is {\it unimodular}\/ if and only if its determinant is $\pm1$. 
For $L=L_{k,m,n}$, it is the case when $m^2-mn+kn=\pm1$, and such parameters are classified by Lemma \ref{claim:240310-01} with $d=\pm1$.

Let us restrict our attention to positive-definite lattices and assume $1\leq m\leq n$.
Then $k$ is necessarily positive, and $L=L_{1,1,1}=\mathbb{Z}$ if $m=n$.
Therefore, it suffices to consider the case with $1\leq m\leq n-1$, whence the integers $p,q$ in Lemma \ref{claim:240310-01} are nonnegative for $d=1$.
The classification is now achieved by the following.
\begin{claim}[Proposition]
\label{claim:240316-02}
\sl
Let $k,m,n$ be integers with $1\leq m\leq n-1$.
Then $L_{k,m,n}$ is positive-definite and unimodular if and only if there exist nonnegative integers $p,q$ satisfying $k^2-1=pq$, $m=k+p$, and $n=2k+p+q$. 
\end{claim}
\begin{claim}[Remark]
\begin{enumerate}
\item[(1)]
The isomorphism $L_{k,m,n}\cong L_{k,n-m,n}$ of opposite lattices corresponds to switching of $p$ and $q$ in Proposition \ref{claim:240316-02}.
\item[(2)]
Assume $1\leq m\leq n-m$.
For $k=1$, the lattice $L_{k,m,n}$ is positive-definite only when $m=1$, and the lattice $L_{1,1,n}\cong \mathbb{Z}^n$ is unimodular. 
For $k\geq 2$, the integers $p,q$ in Proposition \ref{claim:240316-02} must be positive.
\end{enumerate}
\end{claim}
Here are a couple of instances when the lattice $L_{k,m,n}$ becomes positive-definite and unimodular. It is even if and only if $k$ is even. 
\begin{claim}[Example]
\label{claim:240318-04}
\begin{enumerate}
\item[(1)]
The lattice $L_{k,m,n}$ with $m=k+1$ and $n=k^2+2k$ is positive-definite and unimodular for all $k\geq 1$.
It contains a sublattice of type $A_{n}$ formed by the simple roots 
$e_1-e_2,\ldots,e_{n-1}-e_{n},e_1+\cdots+e_{n-1}+2e_n$. 
For $k=2$, the lattice $L_{2,3,8}$ is isomorphic to $E_8$ and it contains a sublattice of type $A_8$.
For $k=4$, the lattice $L_{4,5,24}$ is isomorphic to the Niemeier lattice $N(A_{24})$ of type $A_{24}$.
The lattice $L_{k,k+1,k^2+2k+1}$ is degenerate and $L_{k,k+1,k^2+2k+2}$ is hyperbolic and unimodular. 
\item[(2)]
The lattice $L_{k,m,n}$ with $m=2k-1$ and $n=4k$ is positive-definite and unimodular for all $k\geq 1$.
As it contains a sublattice of type $D_{n}$ formed by the simple roots 
$e_2-e_1,\ldots,e_{n}-e_{n-1},e_1+\cdots+e_{n-2}$, the lattice is isomorphic to $D^+_{n}$. 
For $k=2$, the lattice $L_{2,3,8}$ is isomorphic to $E_8$ and it contains a sublattice of type $D_8$.
For $k=6$, the lattice $L_{6,11,24}\cong D^+_{24}$ is the Niemeier lattice $N(D_{24})$ of type $D_{24}$. 
\end{enumerate}
\end{claim}
\begin{claim}[Note]
As the rank of an even positive-definite unimodular lattice is divisible by $8$,
if $L_{k,m,n}$ is such a lattice, then $k\leq n/4$ follows by Remark \ref{claim:240315-01}.
For example, if $1\leq m\leq n=24$, then $2\leq k\leq 6$ with $k$ even, and we can easily check by Proposition \ref{claim:240316-02} that the lattices $L_{4,5,24}\cong L_{4,19,24}$ and $L_{6,11,24}\cong L_{6,13,24}\cong D^{+}_{24}$ in Example \ref{claim:240318-04} are the only possibilities.
It is not difficult to show that $L_{n/4,n/2-1,n}\cong L_{n/4,n/2+1,n}\cong D^+_{n}$ is the only possibility in case $n$ is a power of $2$.
\end{claim}
%
\subsection{Orthogonal frames and symmetric designs}
\label{subsection:OrthogonalFrames}
%
Let $k,m,n$ be integers with $k\ne 0$ and $1\leq m\leq n-1$, and set $X_{n}=\{1,\ldots,n\}$.
We will call each element of $X_{n}$ a {\it point} and a subset of $X_n$ consisting of $m$ distinct points an {\it $m$-set}. 

For each $m$-set $I$, write $e_I=\sum_{i\in I}e_i$, which is an element of latitude $m$ and belongs to $L=L_{k,m,n}$. 
Then $(e_I|e_J)=|I\cap J|+k-m$ for $m$-sets $I$ and $J$, hence $(e_I|e_J)=0$ if and only if $|I\cap J|=m-k$. 

Let $\mathcal{B}$ be a set of $m$-sets of $X_n$, and let $E(\mathcal{B})$ denote the set of elements of the form $e_{I}$ for $I\in \mathcal{B}$.
If $|I\cap J|=m-k$ for any distinct $I,J\in\mathcal{B}$, then the elements of $E(\mathcal{B})$ are orthogonal to each other with squared norm $k$.
Therefore, under the same assumption and $|\mathcal{B}|=n$, the set $E(\mathcal{B})$ is an {\it orthogonal $k$-frame}\/ of $L$, where $k$ necessarily satisfies $1\leq k\leq m$.

The condition that $|I\cap J|$ is constant for distinct $I,J\in\mathcal{B}$ means that the pair $(X_n,\mathcal{B})$ is the {\it dual} of a {\it $2$-design}, which becomes a {\it symmetric $2$-design}\/ (or a {\it square $2$-design}) if and only if $|\mathcal{B}|=n$. 
The standard parameter of such a symmetric design is given by $2$-$(n,m,\lambda)$, where $\lambda$ is the intersection number $|I\cap J|$ of distinct $I,J\in \mathcal{B}$.

In summary, we have the following. 
\begin{claim}[Proposition]
\label{claim:240326-01}
\sl
Let $k,m,n$ be integers with $1\leq k\leq m\leq n-1$, and let $\mathcal{B}$ be a set of $m$-sets of $X_{n}$.
Then $E(\mathcal{B})$ is an orthogonal $k$-frame of $L_{k,m,n}$ if and only if $(X_n,\mathcal{B})$ is a symmetric $2$-$(n,m,m-k)$ design.
\end{claim}
\begin{claim}[Example]
\label{claim:240319-01}
\rm
Consider the lattice $L_{k,2k-1,4k-1}$ ($k\geq 1$) of determinant $k$, which is a sublattice of $L_{k,2k-1,4k}$ of Example \ref{claim:240318-04} (2) and contains a sublattice of type $A_{4k-1}\subset D_{4k}$. 
If $(X_{4k-1},\mathcal{B})$ is a symmetric $2$-$(4k-1,2k-1,k-1)$ design, then $E(\mathcal{B})$ becomes an orthogonal $k$-frame of $L_{k,2k-1,4k-1}$. 
Existence of such a design is equivalent to that of an Hadamard matrix of order $4k$  (cf.\ \cite{BethEtAl},\cite{Lander}). 
Thus the design exists if $k$ is a power of $2$, for example, and existence for all $k$ is an open problem equivalent to the Hadamard conjecture \cite{Hadamard}. 
\end{claim}
We may generalise the construction of orthogonal subsets as above by replacing $e_I$ with $de_I$ or $m$-sets with $dm$-sets etc., where $d$ is a positive integer and $dm\leq n$ for the latter.
\begin{claim}[Example]
\label{claim:240325-01}
Consider the lattice $L_{k,2k-1,4k}=D_{4k}^{+}$ in Example \ref{claim:240318-04} (2) and set 
$\mathcal{B}=\{X_{4k}\setminus\{2i-1,2i\}\,|\,i=1,\ldots,2k\}$.
Then the elements of $E(\mathcal{B})$ together with $e_1-e_2,e_3-e_4,\ldots,e_{4k-1}-e_{4k}$ form an orthogonal $2$-frame of $D_{4k}^{+}$.
\end{claim}
\begin{claim}[Note]
An integral lattice $L$ admits an orthogonal $k$-frame if and only if it is constructed from a code over $\mathbb{Z}/k\mathbb{Z}$ by (a generalization of) Construction A (cf.\ \cite{ConwaySloane},\cite{HaradaEtAl}).
If $L$ is an irreducible root lattice, it is the case when $L$ is of type $A_1$, $D_n$ ($n\geq 4$ and even), $E_7$, and $E_8$ (cf.\ \cite{Ebeling}).
An orthogonal $2$-frame of $D_n=L_{2,2,n}$ ($n$ even) is easily found to be given by $e_1\pm e_2, \ldots,e_{n-1}\pm e_{n}$, and one for $E_7$ is described by Example \ref{claim:240319-01} by letting $(X_7,\mathcal{B})$ be the Fano plane.
The seven roots of an orthogonal $2$-frame of $E_7$ and the highest root of $E_8$ form an orthogonal $2$-frame of $E_8$. Example \ref{claim:240325-01} with $k=2$ gives an alternative construction of such a frame of $E_8$. 
\end{claim}
%
\def\thesection{\Alph{section}}
\setcounter{section}{0}
\section{Roots of the root lattices}
\label{appendix:TheRoots}
%
We write the shape of an element $x\in L_{k,m,n}\subset\mathbb{Z}^n$ as 
\[
\label{eq:240311-02}
(\cdots (-2)^{d_{-2}}(-1)^{d_{-1}}0^{d_0}1^{d_1}2^{d_2}\cdots)
,\] 
where $d_j$ is the number of $j$'s in the coordinates of $x$.
We often omit $0^{d_0}$. 
For example, $(1^3)$ denotes an element of the form $e_{i_1}+e_{i_2}+e_{i_3}$ with distinct $i_1,i_2,i_3$, or the set of such elements, depending on the context. 

Recall the isomorphism $\theta_{m,n}\colon L_{k,m,n}\to L_{k,n-m,n}$ of opposite lattices (Subsect.\ \ref{240318-09}), which sends the roots of $L_{k,m,n}$ bijectively onto the roots of $L_{k,n-m,n}$.
Therefore, as long as the lists of roots are concerned, we may assume $m\leq n-m$ without loss of generality. 
%
\subsection{Roots of \boldmath$L_{2,m,n}$ ($m=1,2,3$)}
%
%
The symbol $_{n}\mathrm{C}_m$ refers to the binomial coefficient $\binom{n}{m}$.
\paragraph{\boldmath$L_{2,1,n}=A_n$ ($n\geq 1$)}
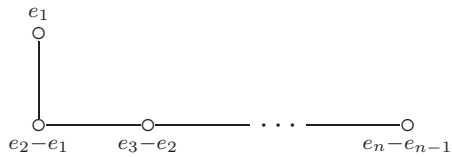
\[
\xymatrix@M0pt@R0pt@C0pt{
\scriptstyle \hskip-2em e_1\hskip-2em\vcenter to 2.5ex{}\cr
\circ\ar@{-}[dd]\cr
\vcenter to 2.5em{}\cr
\circ\ar@{-}[rr]&\hskip3em&\circ\ar@{-}[rr]&\hskip3em&\ \cdots\ \ar@{-}[rr]&\hskip3em&\circ\cr
\scriptstyle \hskip-1em e_2-e_1\hskip-1em&&\scriptstyle \hskip-1em e_3-e_2\hskip-1em &&\scriptstyle  &&\scriptstyle \hskip-2em e_{n}-e_{n-1}\hskip-2em \cr
}
\]
\[
\begin{tabular}{ccll}
latitude&shape&\hfill number\cr
\hline
$\phantom{-}1$&$\phantom{(-1)^6}1^1$&$\hfill {_{n}\mathrm{C}_{1}}=n\phantom{(n-1)}$\cr
$\phantom{-}0$&$(-1)^11^1$&\hfill$2\,{_{n}\mathrm{C}_2}=n(n-1)$\cr
$-1$&$(-1)^1\phantom{1^1}$&\hfill${_{n}\mathrm{C}_1}=n\phantom{(n-1)}$\cr
\hline&&\hfill \llap{total: }$n(n+1)$\rlap{.}
\end{tabular}
\]
%
\paragraph{\boldmath$L_{2,2,n}=D_n$ ($n\geq 2$)}
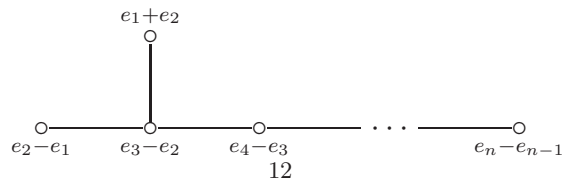
\[
\xymatrix@M0pt@R0pt@C0pt{
&&\scriptstyle \hskip-2em e_1+e_2\hskip-2em\vcenter to 2.5ex{}\cr
&&\circ\ar@{-}[dd]\cr
\vcenter to 2.5em{}\cr
\circ\ar@{-}[rr]&\hskip3em&\circ\ar@{-}[rr]&\hskip3em&\circ\ar@{-}[rr]&\hskip3em&\ \cdots\ \ar@{-}[rr]&\hskip3em&\circ\cr
\scriptstyle \hskip-1em e_2-e_1\hskip-1em&&\scriptstyle \hskip-1em e_3-e_2\hskip-1em &&\scriptstyle \hskip-1em e_4-e_3\hskip-1em&&\scriptstyle  &&\scriptstyle \hskip-2em e_{n}-e_{n-1}\hskip-2em \cr
}
\]
\[
\begin{tabular}{ccll}
latitude&shape&\hfill number\cr
\hline
$\phantom{-}2$&$\phantom{(-1)^6}1^2$&$\hfill {_{n}\mathrm{C}_2}=n(n-1)/2$\cr
$\phantom{-}0$&$(-1)^11^1$&\hfill$2\,{_{n}\mathrm{C}_2}=n(n-1)\phantom{/2}$\cr
$-2$&$(-1)^2\phantom{1^1}$&\hfill${_{n}\mathrm{C}_2}=n(n-1)/2$\cr
\hline&&\hfill \llap{total: }$2n(n-1)\rlap{.}\phantom{/2}$
\end{tabular}
\]
%
%
\paragraph{\boldmath$L_{2,3,n}=E_n$ ($n=6,7,8$)}
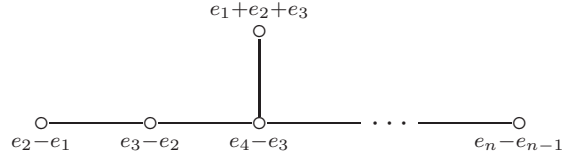
\[
\xymatrix@M0pt@R0pt@C0pt{
&&&&\scriptstyle \hskip-3em e_1+e_2+e_3\hskip-3em\vcenter to 2.5ex{}\cr
&&&&\circ\ar@{-}[dd]\cr
\vcenter to 2.5em{}\cr
\circ\ar@{-}[rr]&\hskip3em&\circ\ar@{-}[rr]&\hskip3em&\circ\ar@{-}[rr]&\hskip3em&\ \cdots\ \ar@{-}[rr]&\hskip3em&\circ\cr
\scriptstyle \hskip-1em e_2-e_1\hskip-1em&&\scriptstyle \hskip-1em e_3-e_2\hskip-1em &&\scriptstyle \hskip-1em e_4-e_3\hskip-1em&&\scriptstyle  &&\scriptstyle \hskip-2em e_{n}-e_{n-1}\hskip-2em \cr
\vcenter to 2ex{}
\cr
}
\]
\paragraph{\boldmath$L_{2,3,6}=E_6$}
\[
\begin{tabular}{ccll}
latitude&shape&\hfil number\cr
\hline
$\phantom{-}6$&$\phantom{(-1)^6}1^6$&$\hfill_6\mathrm{C}_6=\phantom{0}1$\cr
$\phantom{-}3$&$\phantom{(-1)^6}1^3$&$\hfill_6\mathrm{C}_3=20$\cr
$\phantom{-}0$&$(-1)^11^1$&$2\,{_{6}\mathrm{C}_{2}}=30$\cr
$-3$&$(-1)^3\phantom{1^1}$&\hfill$_6\mathrm{C}_3=20$\cr
$-6$&$(-1)^6\phantom{1^1}$&\hfill$_6\mathrm{C}_6=\phantom{0}1$\cr
\hline&&\hfill \llap{total: }$72$\rlap{.}
\cr
\end{tabular}
\]
\paragraph{\boldmath$L_{2,3,7}=E_7$}
\[
\begin{tabular}{ccll}
latitude&shape&\hfill number\cr
\hline
$\phantom{-}6$&$\phantom{(-1)^6}1^6$&$\hfill_7\mathrm{C}_6=\phantom{0}7$\cr
$\phantom{-}3$&$\phantom{(-1)^6}1^3$&$\hfill_7\mathrm{C}_3=35$\cr
$\phantom{-}0$&$(-1)^11^1$&\hfill$2\,{_{7}\mathrm{C}_{2}}=42$\cr
$-3$&$(-1)^3\phantom{1^1}$&\hfill$_7\mathrm{C}_3=35$\cr
$-6$&$(-1)^6\phantom{1^1}$&\hfill$_7\mathrm{C}_6=\phantom{0}7$\cr
\hline&&\hfill \llap{total: }$126$\rlap{.}
\end{tabular}
\]
\paragraph{\boldmath$L_{2,3,8}=E_8$}
\[
\begin{tabular}{ccll}
latitude&shape&\hfill number\cr
\hline
$\phantom{-}9$&$\phantom{(-2)^1}\phantom{(-1)^6}1^72^1$&$\hfill_8\mathrm{C}_1=\phantom{0}8$\cr
$\phantom{-}6$&$\phantom{(-2)^1}\phantom{(-1)^6}1^6\phantom{2^1}$&$\hfill_8\mathrm{C}_6=28$\cr
$\phantom{-}3$&$\phantom{(-2)^1}\phantom{(-1)^6}1^3\phantom{2^1}$&$\hfill_8\mathrm{C}_3=56$\cr
$\phantom{-}0$&$\phantom{(-2)^1}(-1)^11^1\phantom{2^1}$&\hfill$2\,{_{8}\mathrm{C}_{2}}=56$\cr
$-3$&$\phantom{(-2)^1}(-1)^3\phantom{1^1}\phantom{2^1}$&\hfill$_8\mathrm{C}_3=56$\cr
$-6$&$\phantom{(-2)^1}(-1)^6\phantom{1^1}\phantom{2^1}$&\hfill$_8\mathrm{C}_6=28$\cr
$-9$&$(-2)^1(-1)^7\phantom{1^12^1}$&\hfill$_8\mathrm{C}_1=\phantom{0}8$\cr
\hline&&\hfill \llap{total: }$240$\rlap{.}
\end{tabular}
\]
%
\subsection{Orders of the Weyl groups}
%
%
Let $X_n$ denote the root lattice $L_{2,m,n}$ of type $A_n$, $D_n$, and $E_n$ for $m=1$, $2$, and $3$, respectively.
Recall the element $z_n$ given in Remark \ref{claim:240402-01} (2):
\[
z_n=(m-2)(e_1+\cdots+e_n)+(\det X_n)e_n
.\]
Then $(x|z_n)=(\det X_n)x_n$ for $x\in X_n$.
Hence $(x|z_n)=0$ if and only if $x_n=0$, and it follows that the isotropy group of $z_n$ in the Weyl group $W(X_n)$ agrees with the subgroup $W(X_{n-1})$ (cf.\ \cite{Humphreys}).
Thus the orders of $W(X_n)$ are calculated recursively by 
$
|W(X_n)|=c_n\,|W(X_{n-1})|
$, where $c_n$ denotes the cardinality of the $W(X_n)$-orbit of $z_n$.

For $A_n$ and $D_n$, we have $c_n=n+1$ and $c_n=2n$, hence $|W(A_n)|=(n+1)!$ and $|W(D_n)|=2^{n-1}n!$ as expected.

The numbers $c_n$ for $E_n$ and the orders of $W(E_n)$ for $n=6,7,8$ are counted as below; an arrow in the schemes indicates that a reflection with respect to a root of shape $(1^3)$ interchanges a pair of elements of the shapes on both sides of the arrow. 
Note that the action of the subgroup $W(A_{n-1})$ does not change the shapes of elements. 
\paragraph{\boldmath$L_{2,3,6}=E_6$}
\[
|W(E_6)|=27\,|W(D_5)|=27\cdot 1920=51840. 
\]
\[
\begin{tabular}{r}
\xymatrix@R3ex@C1.5em{
z_6\in(1^54^1)
\ar@{<->}[r]
&((-2)^21^4)
\ar@{<->}[r]
&((-2)^51^1)
\hskip-0.4em}
\cr\hline
total: $c_6={_{6}\mathrm{C}_{1}}+{_{6}\mathrm{C}_{2}}+{_{6}\mathrm{C}_{5}}=6+15+6=27$\rlap{.}
\end{tabular}
\]
\paragraph{\boldmath$L_{2,3,7}=E_7$}
\[
|W(E_7)|=56\,|W(E_6)|=56\cdot 51840=2903040.
\]
\[
\begin{tabular}{r}
\xymatrix@R3ex@C1.5em{
z_7\in(1^63^1)
\ar@{<->}[r]&
((-1)^21^5)
\ar@{<->}[r]&
((-1)^51^2)
\ar@{<->}[r]&
((-3)^1(-1)^6)
\hskip-0.4em}
\cr
\hline
total: $c_7={_{7}\mathrm{C}_{1}}+{_{7}\mathrm{C}_{2}}+{_{7}\mathrm{C}_{5}}+{_{7}\mathrm{C}_{6}}=7+21+21+7=56$\rlap{.}
\end{tabular}
\]
\paragraph{\boldmath$L_{2,3,8}=E_8$}
\[
|W(E_8)|=240\,|W(E_7)|=240\cdot 2903040=696729600.
\]
The $W(E_8)$-orbit of $z_8$ is the whole set of the roots since $z_8$ is a root; we have $c_8=240$.

We can count the numbers $c_n$ for $n=4,5$ in the same way as well; we may calculate the orders of $W(E_n)$ ($3\leq n\leq 8$) within the $E_n$ series alone starting with $|W(E_3)|=12$. 
\begin{claim}[Note]
Recall that $w_n=(\det X_n)^{-1}z_n$ is the fundamental weight of $X_n$ corresponding to $\alpha_{n-1}=e_{n}-e_{n-1}$ (Remark \ref{claim:240402-01} (2)).
The $W(E_n)$-orbit of $w_n$ for $n=6$ and $7$ consists of the weights of the minuscule representations of $E_6$ and $E_7$, respectively (cf.\ \cite{Ochiai}). 
For $E_6$, there are two such representations and the weights of the other one constitute the orbit of $-w_6$.
The union of the two orbits for $E_6$ and the orbits for $E_7$ and $E_8$ agree with the sets of minimal vectors of the dual lattices $E_6^*$, $E_7^*$, and $E_8^*=E_8$, respectively, as described in Ref.\ \cite{Shioda}.
The vector $w_n$ is a minuscule weight for $A_n$ ($1\leq n$), $D_n$ ($3\leq n$), and $E_n$ ($4\leq n\leq 7$).
It is a minimal vector of $X_n^*$ for $A_n$ ($1\leq n$), $D_n$ ($4\leq n$), and $E_n$ ($6\leq n\leq 8$).
\end{claim}
%

\end{document}